\documentclass[12pt]{article}
\usepackage{amssymb,latexsym,amsmath,enumerate,verbatim,amsfonts,amsthm}
\usepackage[active]{srcltx}

\newtheorem{Theorem}{Theorem}[section]
\newtheorem{Definition}{Definition}[section]
\newtheorem{Proposition}[Theorem]{Proposition}

\newtheorem{Lemma}{Lemma}[section]

\newtheorem{Question}{Question}[section]

\makeatletter
 \def\@biblabel#1{#1.}
\makeatother

\newcommand{\uu}{{\bf u}}

\newcommand{\x}{{\bf x}}
\newcommand{\y}{{\bf y}}
\newcommand{\z}{{\bf z}}
\newcommand{\q}{{\bf q}}

\newcommand{\0}{{\bf 0}}

\begin{document}
\title{Properties of Some Classes of Structured Tensors\thanks{To appear in: Journal of Optimization: Theory and Applications.}}

\author{Yisheng Song\thanks{ Corresponding author. School of Mathematics and Information Science,
 Henan Normal University, XinXiang HeNan,  P.R. China, 453007.
Email: songyisheng1@gmail.com. This author's work was supported by
the National Natural Science Foundation of P.R. China (Grant No.
11171094, 11271112).  His work was partially done when he was
visiting The Hong Kong Polytechnic University.}, \quad Liqun Qi\thanks{ Department of Applied Mathematics, The Hong Kong Polytechnic University, Hung Hom,
Kowloon, Hong Kong. Email: maqilq@polyu.edu.hk. This author's work was supported by the Hong
Kong Research Grant Council (Grant No. PolyU 502510, 502111, 501212
and 501913).}}

\date{\today}

 \maketitle

%---------------------------------------------------------------------------------Abstract
\begin{abstract}
\noindent  %\vspace{3mm}
  In this paper, we extend some classes of structured matrices to higher order tensors.  We discuss their relationships with positive semi-definite tensors and some other
  structured tensors.    We show that every principal sub-tensor of such a structured tensor is still a structured tensor in the same class, with a lower dimension.
  The potential links of such structured tensors with optimization, nonlinear equations, nonlinear complementarity
problems, variational inequalities and the nonnegative tensor theory
are also discussed.

\noindent {\bf Key words:}\hspace{2mm} P tensor, P$_0$ tensor, B tensor, B$_0$
tensor,  Principal sub-tensor, Eigenvalues.  \vspace{3mm}

\noindent {\bf AMS subject classifications (2010):}\hspace{2mm}
47H15, 47H12, 34B10, 47A52, 47J10, 47H09, 15A48, 47H07.
  \vspace{3mm}

\end{abstract}

%------------------------------------------------------------------------------Section 1

\section{Introduction}
\hspace{4mm}  P and P$_0$ matrices have a long history and wide
applications in mathematical sciences.  Fiedler and Pt\'ak first studied P matrices systematically in
\cite{FiP}.   For the
applications of P and P$_0$ matrices and functions in linear
complementarity problems, variational inequalities and nonlinear
complementarity problems, we refer readers to [2-4]. %\cite{CPS, FaP, QSZ}.
It is well-known that a symmetric matric is a P (P$_0$) matrix if
and only if it is positive (semi-)definite \cite[Pages 147,
153]{CPS}.

On the other hand, motivated by the discussion on positive
definiteness of multivariate homogeneous polynomial forms [5-8], %\cite{BM, HH, JM, WQ},
in 2005, Qi \cite{Qi} introduced the concept of
positive (semi-)definite symmetric tensors. In the same time, Qi  also
introduced eigenvalues, H-eigenvalues, E-eigenvalues and
Z-eigenvalues for symmetric tensors.    It was shown that an even
order symmetric tensor is positive (semi-)definite if and only if
all of its H-eigenvalues or Z-eigenvalues are positive (nonnegative)
\cite[Theorem 5]{Qi}.  Beside automatical control, positive
semi-definite tensors also found applications in magnetic resonance
imaging [10-13] %\cite{CDHS, HHNQ, QYW, QYX}
and spectral hypergraph theory [14-16].
%\cite{HQ, LQY, Qi14}.

The following  questions are natural. Can we extend the concept of P and
P$_0$ matrices to P and P$_0$ tensors?  If this can be done, is it
true a symmetric tensor is a P (P$_0$) tensor if and only if it is
positive (semi-)definite?   Are there any odd order P (P$_0$)
tensors?

In Section 3, we will extend the concept of P and P$_0$
matrices to P and P$_0$ tensors.  We will show that a symmetric
tensor is a P (P$_0$) tensor if and only if it is positive (semi-)
definite. The close relationship between P (P$_0$) tensors and
positive (semi-)definite tensors justifies our research on P and
P$_0$ tensors. We will show that there does not exists an odd order
symmetric P tensor.   If an odd order non-symmetric P tensor exists,
then it has no Z-eigenvalues. An odd order P$_0$ tensor has no
nonzero Z-eigenvalues.

In Section 4, we will further study some properties of P and P$_0$
tensors.  We will show that every principal sub-tensor of a P
(P$_0$) tensor is still a P (P$_0$) tensor, and give some sufficient
and necessary conditions for a tensor to be a P (P$_0$) tensor.

The class of B matrices is a subclass of P matrices \cite{Pe, Pe1}.
We will extend the concept of B matrices to B and B$_0$ tensors in
Section 5. It is easily checkable if a given tensor is a B or B$_0$
tensor or not. We will show that a Z tensor is diagonally dominated if
and only if it is  a B$_0$ tensor.   It was proved in
\cite{ZQZ} that a diagonally dominated Z tensor is an M tensor.
Laplacian tensors of uniform hypergraphs, defined as a natural
extension of Laplacian matrices of graphs, are M tensors [16, 20-23]. %\cite{Qi14, HQ14, HQS, HQX, QSW}.
This justifies our research on B and B$_0$
tensors.

Some final remarks will be given in Section 6.  The potential links
of P, P$_0$, B and B$_0$ tensors with optimization, nonlinear
equations, nonlinear complementarity problems, variational
inequalities and the nonnegative tensor theory are discussed. These
encourage further research on P, P$_0$, B and B$_0$ tensors.

\section{Preliminaries}

In this section, we will define the notations and collect some basic definitions and facts, which will be
used later on.

Denote $I_n := \{ 1,2, \cdots, n \}$ and $\mathbb{R}^n:=\{(x_1, x_2,\cdots, x_n)^T;x_i\in  \mathbb{R},  i\in I_n$, where $\mathbb{R}$ is the set of real numbers.  The definitions of P and P$_0$
matrices are as follows.

\begin{Definition} \label{d1}\em
Let $A = (a_{ij})$ be an $n \times n $ real matrix.   We say that
$A$ is\\
(i) a P$_0$ matrix iff for any nonzero vector $\x$ in $\mathbb{R}^n$, there
exists $i \in I_n$ such that $x_i \not = 0$ and
$$ x_i (Ax)_i \ge 0;$$
(ii) a P matrix iff for any nonzero vector $\x$ in $\mathbb{R}^n$,
$$\max_{i\in I_n}x_i (Ax)_i > 0.$$
\end{Definition}

A real $m$th order $n$-dimensional tensor (hypermatrix) $\mathcal{A} = (a_{i_1\cdots i_m})$ is a multi-array of real entries $a_{i_1\cdots
i_m}$, where $i_j \in I_n$ for $j \in I_m$. Denote the set of all
real $m$th order $n$-dimensional tensors by $T_{m, n}$. Then $T_{m,
n}$ is a linear space of dimension $n^m$. Let $\mathcal{A} = (a_{i_1\cdots
i_m}) \in T_{m, n}$. If the entries $a_{i_1\cdots i_m}$ are
invariant under any permutation of their indices, then $\mathcal{A}$ is
called a {\bf symmetric tensor}.  Denote the set of all real $m$th
order $n$-dimensional tensors by $S_{m, n}$. Then $S_{m, n}$ is a
linear subspace of $T_{m, n}$.    Let $\mathcal{A} =(a_{i_1\cdots i_m}) \in
T_{m, n}$ and $\x \in \mathbb{R}^n$.   Then $\mathcal{A} x^m$ is a homogeneous
polynomial of degree $m$, defined by
$$\mathcal{A} x^m: = \sum_{i_1,\cdots, i_m=1}^n a_{i_1\cdots i_m}x_{i_1}\cdots
x_{i_m}.$$   A tensor $\mathcal{A} \in T_{m, n}$ is called {\bf positive
semi-definite} if for any vector $\x \in \mathbb{R}^n$, $\mathcal{A} \x^m \ge 0$,
and is called {\bf positive definite} if for any nonzero vector $\x
\in \mathbb{R}^n$, $\mathcal{A} \x^m > 0$.   Clearly, if $m$ is odd, there is no
nontrivial positive semi-definite tensors.\\

In the following, we extend the definitions of eigenvalues,
H-eigenvalues, E-eigenvalues and Z-eigenvalues of tensors in $S_{m,
n}$ in \cite{Qi} to tensors in $T_{m, n}$.

Denote ${\mathbb{C}}^n:=\{(x_1,x_2,\cdots,x_n)^T;x_i\in {\mathbb{C}},  i \in I_n \}$, where ${\mathbb{C}}$ is the set of complex numbers.
 For any vector $\x \in
\mathbb{C}^n$, $\x^{[m-1]}$ is a vector in $\mathbb{C}^n$ with
its $i$th component defined as $x_i^{m-1}$ for $i \in I_n$.  Let $\mathcal{A}
\in T_{m, n}$.  If  and only if there is a nonzero vector $\x \in \mathbb{C}^n$ and a number $\lambda \in \mathbb{C}$ such that
\begin{equation} \label{eig}
\mathcal{A} \x^{m-1} = \lambda \x^{[m-1]},
\end{equation}
then $\lambda$ is called an {\bf eigenvalue} of $\mathcal{A}$ and $\x$ is
called an {\bf eigenvector} of $\mathcal{A}$, associated with $\lambda$.   If
the eigenvector $\x$ is real, then the eigenvalue $\lambda$ is also
real.  In this case, $\lambda$ and $\x$ are called an {\bf
H-eigenvalue} and an {\bf H-eigenvector} of $\mathcal{A}$, respectively.
For an even order symmetric tensor, H-eigenvalues always exist. An
even order symmetric tensor is positive (semi-)definite if and only
if all of its H-eigenvalues are positive (nonnegative).  Let $\mathcal{A} \in
T_{m, n}$.  If and only if there is a nonzero vector $\x \in \mathbb{C}^n$
and a number $\lambda \in \mathbb{C}$ such that
\begin{equation} \label{eig1}
\mathcal{A} \x^{m-1} = \lambda \x,  \ \ \x^\top \x = 1,
\end{equation}
then $\lambda$ is called an {\bf E-eigenvalue} of $\mathcal{A}$ and $\x$ is
called an {\bf E-eigenvector} of $\mathcal{A}$, associated with $\lambda$. If
the E-eigenvector $\x$ is real, then the E-eigenvalue $\lambda$ is
also real.  In this case, $\lambda$ and $\x$ are called an {\bf
Z-eigenvalue} and an {\bf Z-eigenvector} of $\mathcal{A}$, respectively. For
a symmetric tensor, H-eigenvalues always exist. An even order
symmetric tensor is positive (semi-)definite if and only if all of
its H-eigenvalues or Z-eigenvalues are positive (nonnegative)
\cite[Theorem 5]{Qi}.\\

Throughout this paper, we assume that $m \ge 2$ and $n \ge 1$.   We
use small letters $x, u, v, \alpha, \cdots$, for scalars, small bold
letters $\x, \y, \uu, \cdots$, for vectors, capital letters $A, B,
\cdots$, for matrices, calligraphic letters $\mathcal{A}, \mathcal{B}, \cdots$, for
tensors. All the tensors discussed in this paper are real.   We
denote the zero tensor in $T_{m, n}$ by $\mathcal{O}$.

\section{P and P$_0$ Tensors}
\hspace{4mm}

Let $\mathcal{A} = (a_{i_1\cdots i_m}) \in T_{m, n}$ and $\x \in\mathbb{R}^n$.   Then $\mathcal{A} \x^{m-1}$ is a vector in $\mathbb{R}^n$ with
its $i$th component as
$$\left(\mathcal{A} \x^{m-1}\right)_i: = \sum_{i_2, \cdots, i_m=1}^n a_{ii_2\cdots
i_m}x_{i_2}\cdots x_{i_m}$$ for $i \in I_n$.   We now give the
definitions of P and P$_0$ tensors.

\begin{Definition} \label{d2}\em
Let $\mathcal{A}  = (a_{i_1\cdots i_m}) \in T_{m, n}$.   We say that $A$ is\\
(i) a P$_0$ tensor iff for any nonzero vector $\x$ in $\mathbb{R}^n$, there
exists $i \in I_n$ such that $x_i \not = 0$ and
$$ x_i \left(\mathcal{A} \x^{m-1}\right)_i \ge 0;$$
(ii) a P tensor iff for any nonzero vector $\x$ in $\mathbb{R}^n$,
$$\max_{i\in I_n}x_i \left(\mathcal{A} \x^{m-1}\right)_i > 0.$$
\end{Definition}

Clearly, this definition is a natural extension of Definition \ref{d1}.

We first prove a proposition.

\begin{Proposition} \label{p1}\em
Let $\mathcal{A} \in S_{m, n}$.   If $\mathcal{A} \x^m = 0$ for all $\x \in \mathbb{R}^n$,
then $\mathcal{A} = \mathcal{O}$.
\end{Proposition}
\begin{proof}   Denote $f(\x) = \mathcal{A} \x^m$.   Then $f(\x) = 0$ for all $\x \in \mathbb{R}^n$.  This implies
all the partial derivatives of $f$ are zero.   Since the entries of
$\mathcal{A}$ are just some higher order partial derivatives of $f$, we see
that $\mathcal{A} = \mathcal{O}$.
\end{proof}

We now have the following theorem.

\begin{Theorem}\label{t1}\em
Let $\mathcal{A} \in T_{m, n}$ be a P (P$_0$) tensor. Then when $m$ is even,
all of its H-eigenvalues and Z-eigenvalues of $\mathcal{A}$ are positive
(nonnegative). A symmetric tensor is a P (P$_0$) tensor if and only
if it is positive (semi-)definite.  There does not exist an odd
order symmetric P tensor. If an odd order nonsymmetric P tensor
exists, then it has no Z-eigenvalues. An odd order P$_0$ tensor has
no nonzero Z-eigenvalues.
 \end{Theorem}

 \begin{proof}
 Let $m$ be even and an H-eigenvalue $\lambda$ of $\mathcal{A}$ be given. If $\mathcal{A}$ is a P tensor, then by the definition of H-eigenvalues, there is a nonzero $\x \in \mathbb{R}^n$ and a number $\lambda \in \Re$ such that (\ref{eig}) holds. Then by the definition of P tensors, there exists $i\in I_n$ such that
 $$0< x_i(\mathcal{A} \x^{m-1})_i=\lambda x_i^{m}.$$ Since $m$ is an even number, we have $\lambda>
 0$.  Similarly, if $\mathcal{A}$ is a P$_0$ tensor, we may prove that
 $\lambda \ge 0$.   By \cite[Theorem 5]{Qi}, if all H-eigenvalues of an even order symmetric tensor are positive (nonnegative), then that tensor is positive
 (semi-)definite.   We see now that an even order symmetric tensor
 is a P (P$_0$) tensor only if it is positive
(semi-)definite.   By the definitions of P (P$_0$) tensors and
positive (semi-)definite tensors, it is easy to see that an even
order symmetric tensor is a P (P$_0$) tensor if it is positive
(semi-)definite.    Thus, an even order symmetric tensor is a P
(P$_0$)tensor if and only if it is positive (semi-)definite.

 Now, let an Z-eigenvalue $\lambda$ of $\mathcal{A}$ be
 given.  If $\mathcal{A}$ is a P tensor, then by the definition of Z-eigenvalues, there is an $\x \in \mathbb{R}^n$ and and a number $\lambda \in \Re$ such that (\ref{eig1}) holds.  Then by the definition of P tensors, there exists an $i\in I_n$ such that
 $$0<x_i(\mathcal{A} \x^{m-1})_i=\lambda x_i^2.$$   Thus
 $\lambda>0$.     Note that for this, we do not assume that $m$ is
 even.   However, when $m$ is odd, if $\lambda$ is a Z-eigenvalue
 of a tensor in $T_{m, n}$ with a Z-eigenvector
 $\x$, by the definition of Z-eigenvalues, $-\lambda$ is also a Z-eigenvalue
 of that tensor with an Z-eigenvector
 $-\x$.   Thus, if an odd order P tensor
exists, then it has no Z-eigenvalues.    However, by \cite[Theorem 5]{Qi}, a
symmetric tensor always has Z-eigenvalues.    Thus, an odd order
symmetric P-tensor does not exist.  Since an odd order symmetric positive definite tensor
also does not exist and an even order symmetric tensor is a P
tensor if and only if it is positive definite, we conclude that a symmetric tensor is a P
tensor if and only if it is positive definite.

Similarly, if $\mathcal{A}$ is a P$_0$
tensor, we may prove that all of its Z-eigenvalues are nonnegative.
When $m$ is odd, this also means that all of its Z-eigenvalues are
non-positive.   Thus, an odd order P$_0$ tensor has no nonzero
Z-eigenvalues.   By \cite[Theorem 5]{Qi}, a
symmetric tensor always has Z-eigenvalues.  Thus, both the largest Z-eigenvalue $\lambda_{\max}$ and the smallest
Z-eigenvalue $\lambda_{\min}$  of an odd order symmetric P$_0$ tensor $\mathcal{A}$ are zero.  By \cite[Theorem 5]{Qi},
we have
$$\lambda_{\max} = \max \{ \mathcal{A} \x^m : \x^\top \x = 1 \}$$
and
$$\lambda_{\min} = \min \{ \mathcal{A} \x^m : \x^\top \x = 1 \}.$$
Thus, if $\mathcal{A}$ is an odd order symmetric P$_0$ tensor, $\mathcal{A} \x^m = 0$
for all $\x \in \mathbb{R}^n$.  By Proposition \ref{p1}, this implies that
$\mathcal{A} = \mathcal{O}$. By the definition of positive semi-definite tensors, if
$\mathcal{A}$ is an odd order symmetric positive semi-definite tensor, then
$\mathcal{A} = \mathcal{O}$. Since an even order symmetric tensor is a P$_0$ tensor
if and only if it is positive semi-definite, we conclude that a
symmetric tensor is a P$_0$ tensor if and only if it is positive
semi-definite.
 The theorem is proved.
 \end{proof}

\begin{Question} \label{Q1}\em Is there an odd order non-symmetric P tensor?  Is there an odd order nonzero nonsymmetric P$_0$ tensor?\end{Question}

\section{Properties of P and P$_0$ Tensors}
\hspace{4mm}   In this section, we will study some properties of P
and P$_0$ tensors. Based on the definition of P matrices, Mathias
and Pang \cite{MP} introduced a fundamental quantity
 $\alpha(A)$ corresponding to a P matrix $A$ by
   \begin{equation}\label{alpha}\alpha(A):=\min_{\|\x \|_\infty=1}\left\{\max_{i\in I_n}x_i(A\x)_i\right\}\end{equation}
and studied its properties and applications. Mathias \cite{Ma}
showed that $\alpha(A)$ has a lower bound that is larger than $0$
whenever $A$ is a P matrix. Xiu and Zhang \cite{XZ} gave some
extensions of such a quantity and obtained global error bounds for
the vertical and horizontal linear complementarity problems. Also see   Garc\'ia-Esnaola, and Pe\~na \cite{GP} for the error bounds for linear complementarity problems of B-matrices.

In the following, we will show that every principal sub-tensor of a
P (P$_0$) tensor is still a P (P$_0$) tensor, and give some
sufficient and necessary conditions for a tensor to be a P tensor.
Let $\mathcal{A} \in T_{m, n}$.  Define an operator $T_\mathcal{A} : \mathbb{R}^n \to \mathbb{R}^n$
by for any $\x \in \mathbb{R}^n$,
\begin{equation} \label{TA}
T_\mathcal{A} (\x) := \begin{cases}\|\x \|_2^{2-m}\mathcal{A} \x^{m-1},\ \x \neq\0\\
\0,\ \ \ \ \ \ \ \ \ \ \  \ \ \ \ \ \ \ \
\x =\0.\end{cases}
\end{equation}
When $m$ is even, define another operator $F_\mathcal{A} : \mathbb{R}^n \to \mathbb{R}^n$ by for any $\x \in \mathbb{R}^n$,
\begin{equation} \label{FA}
F_\mathcal{A} (\x) :=(\mathcal{A}  \x^{m-1})^{\left[\frac1{m-1}\right]}.
\end{equation}
Here, for a vector $\y \in \mathbb{R}^n$, $\y^{\left[\frac1{m-1}\right]}$ is a vector
in $\mathbb{R}^n$, with its $i$th component to be $y_i^{\frac1{m-1}}$. When
$m$ is even, this is well defined. Then we define two quantities
\begin{equation} \label{Talpha}
\alpha(T_\mathcal{A}):=\min_{\|\x \|_\infty=1}\max_{i \in I_n}x_i(T_\mathcal{A}(\x))_i
\end{equation}
for any $m$, and
\begin{equation} \label{Falpha}
\alpha(F_\mathcal{A}):=\min_{\|\x \|_\infty=1}\max_{i \in I_n}x_i(F_\mathcal{A}(\x))_i
\end{equation}
 when $m$ is even.   When $m=2$, $\alpha(T_\mathcal{A})$ and $\alpha(F_\mathcal{A})$ are simply $\alpha(A)$ defined by (\ref{alpha}).
 We will establish monotonicity  and boundedness of such two quantities when $\mathcal{A}$ is a P (P$_0$) tensor.
 Furthermore, we will show that $\mathcal{A}$ is a P (P$_0$) tensor if and only if $\alpha(T_\mathcal{A})$ is positive (nonnegative),  and when $m$ is even, $\mathcal{A}$ is a P tensor (P$_0$) if and only if $\alpha(F_\mathcal{A})$ is positive (nonnegative).

\subsection{Principal Sub-Tensors of P (P$_0$) Tensors}
\hspace{4mm}

%%%%%%%%%%%%%%%%%%%%%%%%%%%%%%%%%%%%%%%%%%%%%%%%%%%%%%%%%%%%%%%%%%%%%%%%%%%%%%%%%%%%%%%%%%%%%%%%%%%%%%%%%%%%%%
  Recall that a tensor $\mathcal{C} \in T_{m, r}$  is called {\bf a principal sub-tensor}  of a tensor $\mathcal{A} = (a_{i_1\cdots i_m}) \in T_{m, n}$ ($1 \le r\leq n$) iff there is a set $J$ that composed of $r$ elements in $I_n$ such that
 $$\mathcal{C} = (a_{i_1\cdots i_m}),\mbox{ for all } i_1, i_2, \cdots, i_m\in J.$$ The concept were first introduced and used in \cite{Qi} for symmetric tensor. We denote by $\mathcal{A}^J_r$ the principal sub-tensor of a tensor $\mathcal{A} \in T_{m, n}$ such that the entries of  $\mathcal{A}^J_r$ are indexed by $J \subset I_n$ with $|J|=r$ ($1 \le r\leq n$), and denote by $\x_J$ the $r$-dimensional sub-vector of a vector $\x \in \mathbb{R}^n$, with the components of $\x_J$ indexed by $J$.   Note that for $r=1$, the principal sub-tensors are just the diagonal entries.
 %%%%%%%%%%%%%%%%%%%%%%%%%%%%%%%%%%%%%%%%%%%%%%%%%%%%%%%%%%%%%%%%%%%%%%%%%%%%%%%%%%%%%%%%%%%%%%%%%%%%%%%%%%%%%%

 \begin{Theorem}\label{t2} \em Let $\mathcal{A}$ be a P (P$_0$) tensor.  Then every principal sub-tensor of $\mathcal{A}$ is P (P$_0$) tensor.   In particular, all the diagonal entries of
 a P (P$_0$) tensor are positive (nonnegative).
  \end{Theorem}

  \begin{proof}
   Let a principal sub-tensor $\mathcal{A}^J_r$ of $\mathcal{A}$ be given. Then for each nonzero vector $\x=(x_{j_1},\cdots,x_{j_r})^\top \in \Re^r$, we may choose $\x^*=(x^*_1,x^*_2, \cdots, x^*_n)^\top \in \mathbb{R}^n$ with $x^*_i=x_i $ for $ i\in J$ and $x^*_i=0$ for $i\notin J$. Suppose now that $\mathcal{A}$ is a P tensor, then there exists $j\in I_n$ such that
  $$0<x^*_j(\mathcal{A}(\x^*)^{m-1})_j=x_j(\mathcal{A}^J_r\x_J^{m-1})_j.$$ By the definition of $\x^*$, we have $j\in J$, and hence  $\mathcal{A}^J_r$ is a P tensor.  The case for P$_0$ tensors can be proved similarly.
   \end{proof}

\subsection{A Necessary and Sufficient Condition for P Tensors}
\hspace{4mm}

The following is a sufficient and necessary condition for a tensor to be a P tensor.

 \begin{Theorem}\label{t3}\em Let $\mathcal{A} \in T_{m, n}$.   Then $\mathcal{A}$ is a P tensor if and only if for each nonzero $\x\in \mathbb{R}^n$, there exists an $n-$dimensional positive diagonal matrix $D_{\x}$ such that $\x^\top D_{\x}(\mathcal{A} \x^{m-1})$ is positive.
  \end{Theorem}

   \begin{proof}
  First, we show the necessity.   Take a nonzero $\x \in \mathbb{R}^n$.  It follows from the definition of  P tensors that there is $k\in I_n$ such that $x_k(\mathcal{A} \x^{m-1})_k>0$. Then for an enough small $\varepsilon>0$, we have $$x_k(\mathcal{A} \x^{m-1})_k+\varepsilon\left(\sum_{j \in I_n, j\ne k}x_j\left(\mathcal{A} \x^{m-1}\right)_j\right)>0.$$
 Take $D_{\x}=\mbox{diag}(d_1,d_2,\cdots,d_n)$ with $d_k=1$ and $d_j=\varepsilon$ for $j\ne k$. Then we have $$\x^\top D_{\x}(\mathcal{A} \x^{m-1})>0.$$

   Now we show the sufficiency. Assume that for each nonzero $\x \in \mathbb{R}^n$, there exists an $n-$dimensional  diagonal matrix $D_{\x}=\mbox{diag}(d_1,d_2,\cdots,d_n)$ with $d_i>0$ for all $i \in I_n$ such that $$0<\x^\top D_{\x}(\mathcal{A} \x^{m-1})=\sum_{i=1}^{n}d_i(x_i(\mathcal{A} \x^{m-1})_i).$$
  Since $d_i>0$ for all $i\in I_n$, there is an $i\in I_n$ such that $x_i(\mathcal{A} \x^{m-1})_i>0$. Otherwise, $x_i(\mathcal{A} \x^{m-1})_i\leq0$ for all $i$.  Then $\sum_{i=1}^{n}d_i(x_i(\mathcal{A} \x^{m-1})_i)\leq 0$, a contradiction.  Hence $\mathcal{A}$ is a P tensor.

 The desired conclusion follows.  \end{proof}

\subsection{Monotonicity and Boundedness of $\alpha(F_\mathcal{A})$ and $\alpha(T_\mathcal{A})$}
\hspace{4mm}

Recall that an operator $T : \mathbb{R}^n \to \mathbb{R}^n$ is called {\bf
positively homogeneous} iff $T(t\x)=tT(\x)$ for each $t>0$ and all
$\x\in \mathbb{R}^n$.  For $\x \in \mathbb{R}^n$, it is known well that
$$\|\x\|_\infty:=\max\{|x_i|;i \in I_n \}\mbox{ and } \|\x
\|_2:=\left(\sum_{i=1}^n|x_i|^2\right)^{\frac12}$$ are two main norms
defined on $\mathbb{R}^n$.  Then for a continuous, positively homogeneous
operator $T : \mathbb{R}^n\to \mathbb{R}^n$,  it is obvious that
$$\|T\|_\infty:=\max_{\|\x\|_\infty=1}\|T(\x)\|_\infty$$ is an
operator norm of $T$ and
$\|T(\x)\|_\infty\leq\|T\|_\infty\|\x\|_\infty$ for any $\x \in
\mathbb{R}^n$.  For $\mathcal{A} \in T_{m, n}$, let $T_\mathcal{A}$ be defined by (\ref{TA}).
When $m$ is even, let $F_\mathcal{A}$ be defined by (\ref{FA}). Clearly, both
$F_\mathcal{A}$ and $T_\mathcal{A}$ are continuous and positively homogeneous. The
following upper bounds of the operator norm were established  by
Song and Qi \cite{SQ}.
%%%%%%%%%%%%%%%%%%%%%%%%%%%%%%%%%%%%%%%%%%%%%%%%%%%%%%%%%%%%%%%%%%%%%%%%%%%%%%%%%%%%%%%%%%%%%%%%%%%%%%%%%%%%%%
\begin{Lemma} \label{l1} \em (Song and Qi \cite[Theorem 4.3]{SQ})  Let $\mathcal{A} = (a_{i_1\cdots i_m}) \in T_{m, n}$. Then
\begin{itemize}
\item[(i)] $\|T_\mathcal{A}\|_{\infty}\leq \max\limits_{i\in I_n}\left(\sum\limits_{i_2,\cdots,i_m=1}^{n}|a_{ii_2\cdots i_m}|\right)$;
\item[(ii)] $\|F_\mathcal{A}\|_{\infty}\leq\max\limits_{i\in I_n}\left(\sum\limits_{i_2,\cdots,i_m=1}^{n}|a_{ii_2\cdots i_m}|\right)^{\frac{1}{m-1}}$, when $m$ is even.
\end{itemize} \end{Lemma}

Let $\alpha(F_\mathcal{A})$ and $\alpha(T_\mathcal{A})$ be defined by (\ref{Falpha}) and (\ref{Talpha}).     We now establish their monotonicity and boundedness.   The proof technique is similar to the proof technique of \cite[Theorem 1.2]{XZ}.   For completeness, we give the proof here.

\begin{Theorem}\label{p2}\em Let  $\mathcal{D} =\mbox{diag}(d_1, d_2,\cdots, d_n)$ be a nonnegative diagonal tensor in $T_{m, n}$ and $\mathcal{A} = (a_{i_1\cdots i_m})$ be a P$_0$ tensor in $T_{m, n}$. Then
 \begin{itemize}
\item[(i)] $\alpha(T_\mathcal{A})\leq\alpha(T_{\mathcal{A}+\mathcal{D}})$  whenever  $m$ is even;
\item[(ii)]  $\alpha(T_\mathcal{A})\leq\alpha(T_{\mathcal{A}^J_r})$ for all principal sub-tensors $\mathcal{A}^J_r$;
\item[(iii)]  $\alpha(F_\mathcal{A})\leq\alpha(F_{\mathcal{A}^J_r})$ for all principal sub-tensors $\mathcal{A}^J_r$, when $m$ is even;
\item[(iv)]  $\alpha(T_\mathcal{A})\leq \max\limits_{i\in I_n}\left(\sum\limits_{i_2,\cdots,i_m=1}^{n}|a_{ii_2\cdots i_m}|\right)$;
\item[(v)]  $\alpha(F_\mathcal{A})\leq \max\limits_{i\in I_n}\left(\sum\limits_{i_2,\cdots,i_m=1}^{n}|a_{ii_2\cdots i_m}|\right)^{\frac1{m-1}}$, when $m$ is even.
\end{itemize}
  \end{Theorem}

  \begin{proof}
  (i) By the definition of P$_0$ tensors, clearly $\mathcal{A}+\mathcal{D}$ is a P$_0$ tensor.  Then $\alpha(T_{\mathcal{A}+\mathcal{D}})$ is well-defined. Since $m$ is even, then $x_i^m>0$ for $x_i\ne 0$, and so
  $$\begin{aligned}
  \alpha(T_\mathcal{A})=&\min_{\|\x\|_\infty=1}\left\{\max_{i\in I_n}x_i(T_\mathcal{A}(\x))_i\right\}\\
  =& \min_{\|\x\|_\infty=1}\left\{\|\x\|_2^{2-m}\max_{i\in I_n}x_i(\mathcal{A} \x^{m-1})_i\right\}\\
  \leq &\min_{\|\x\|_\infty=1}\left\{\|\x\|_2^{2-m}\max_{i\in I_n}\left\{x_i(\mathcal{A} \x^{m-1})_i+d_ix_i^m\right\}\right\}\\
  =& \min_{\|\x\|_\infty=1}\left\{\max_{i\in I_n}x_i\left(\|\x\|_2^{2-m}(\mathcal{A}+\mathcal{D})\x^{m-1}\right)_i\right\}\\
  =&\min_{\|\x\|_\infty=1}\left\{\max_{i\in I_n}x_i\left(T_{\mathcal{A}+\mathcal{D}}(\x)\right)_i\right\}\\
  =&\alpha(T_{\mathcal{A}+\mathcal{D}}).
  \end{aligned}$$

  (ii) Let a principal sub-tensor $\mathcal{A}^J_r$ of $\mathcal{A}$ be given. Then for each nonzero vector $\z=(z_1,\cdots,z_r)^\top \in \Re^r$, we may define $\y(\z)=(y_1(\z),y_2(\z), \cdots, y_n(\z))^\top \in \mathbb{R}^n$ with $y_i(\z)=z_i $ for $ i\in J$ and $y_i(\z)=0$ for $i\notin J$.  Thus $\|\z\|_\infty=\|\y(\z)\|_\infty$ and $\|\z\|_2=\|\y(\z)\|_2$. Hence,
  $$\begin{aligned}
  \alpha(T_\mathcal{A})=&\min_{\|\x\|_\infty=1}\left\{\max_{i\in I_n}x_i(T_\mathcal{A}(\x))_i\right\}\\
  =&\min_{\|\x\|_\infty=1}\left\{ \|\x\|_2^{2-m}\max_{i\in I_n}x_i(\mathcal{A}\x^{m-1})_i\right\}\\
  \leq & \min_{\|\y(\z)\|_\infty=1}\left\{\|\y(\z)\|_2^{2-m}\max_{i\in I_n}\{\y(\z)_i(\mathcal{A}\y(\z)^{m-1})_i\right\}\}\\
  =& \min_{\|\z\|_\infty=1}\left\{\max_{i\in I_n}z_i(\|\z\|_2^{2-m}\mathcal{A}^J_r\z^{m-1})_i\right\}\\
  =&\min_{\|\z\|_\infty=1}\left\{\max_{i\in I_n}z_i(T_{\mathcal{A}^J_r}(\z))_i\right\}\\
  =&\alpha(T_{\mathcal{A}^J_r}).
  \end{aligned}$$

 Similarly, we may show (iii).

  (iv) Since for each nonzero vector $\x=(x_1,\cdots,x_n)^\top \in \mathbb{R}^n$ and each $i\in I_n$,
  $$x_i(T_\mathcal{A}(\x))_i \leq \|\x\|_\infty \|T_\mathcal{A}(\x)\|_\infty\leq\|T_\mathcal{A}\|_\infty\|\x\|_\infty^2,$$
  Then $$\max_{i\in I_n}x_i(T_\mathcal{A}(\x))_i\leq\|T_\mathcal{A}\|_\infty\|\x\|^2_\infty.$$
  Therefore, we have $$\alpha(T_\mathcal{A})=\min_{\|\x\|_\infty=1}\{\max_{i\in I_n}x_i(T_\mathcal{A}(\x))_i\}\leq \|T_\mathcal{A}\|_\infty,$$
  and hence, by Lemma \ref{l1}, the desired conclusion follows.

  Similarly, we may show (v).
  \end{proof}

\subsection{Necessary and Sufficient  Conditions Based Upon $\alpha(F_\mathcal{A})$ and $\alpha(T_\mathcal{A})$}
\hspace{4mm}

We now give necessary  and sufficient conditions for a tensor $A \in T_{m, n}$ to be a P (P$_0$) tensor, based upon $\alpha(F_\mathcal{A})$ and $\alpha(T_\mathcal{A})$.

 \begin{Theorem}\label{t4}\em Let $\mathcal{A} \in T_{m, n}$.  Then
 \begin{itemize}
\item[(i)] $\mathcal{A}$ is a P (P$_0$) tensor if and only if $\alpha(T_\mathcal{A})$ is positive (nonnegative);
\item[(ii)] when $m$ is even, $\mathcal{A}$ is a P (P$_0$) tensor if and only if $\alpha(F_\mathcal{A})$ is positive (nonnegative).
\end{itemize}
  \end{Theorem}

  \begin{proof}  We only prove the case for P tensors.  The proof for the P$_0$ tensor case is similar.

   (i) Let $\mathcal{A}$ be a P tensor. Then it follows from the definition of P tensors that for each nonzero $\x \in \mathbb{R}^n$, $$\max_{i\in I_n}x_i(\mathcal{A} \x^{m-1})_i>0,$$
   and so $$\max_{i\in I_n}x_i(\|\x\|_2^{2-m}\mathcal{A} \x^{m-1})_i=\|\x\|_2^{2-m}\max_{i\in I_n}x_i(\mathcal{A} \x^{m-1})_i>0.$$ Therefore we have $$\alpha(T_\mathcal{A})=\min_{\|\x \|_\infty=1}\left\{\max_{i\in I_n}x_i(T_\mathcal{A}(\x))_i\right\}>0.$$

   If $\alpha(T_\mathcal{A})>0$, then it is obvious that for each nonzero $\y \in \mathbb{R}^n$,
   $$\max_{i\in I_n}\left(\frac{\y}{\|\y\|_\infty}\right)_i\left(T_\mathcal{A}\left(\frac{\y}{\|\y\|_\infty}\right)\right)_i\geq\alpha(T_\mathcal{A}) >0.$$ Hence, $$\max_{i\in I_n}y_i(T_\mathcal{A}(\y))_i= \max_{i\in I_n}y_i\left(\|\y\|_2^{2-m}\mathcal{A}\y^{m-1}\right)_i >0.$$
 Thus   $y_j(\mathcal{A}\y^{m-1})_j>0$ for some $j\in I_n$, i.e., $\mathcal{A}$ is a P tensor.

 (ii)  Assume that $m$ is even.

 Let $\mathcal{A}$ be a P tensor. Then for each nonzero $\x\in \mathbb{R}^n$,  there exists an $i\in I_n$ such that $x_i(\mathcal{A} \x^{m-1})_i>0,$  and so
 $$0<x_i^{\frac1{m-1}}(\mathcal{A}\x^{m-1})_i^{\frac1{m-1}}=x_i^{\frac{2-m}{m-1}}(x_i(\mathcal{A}\x^{m-1})_i^{\frac1{m-1}}).$$
 Since $m$ is even, we have $x_i^{\frac{2-m}{m-1}}>0$ for $x_i\ne0$, and so,
 $$0<x_i(\mathcal{A}\x^{m-1})^{\frac1{m-1}}_i=x_i(F_\mathcal{A}(\x))_i.$$
 That is,  for each nonzero $\x\in \mathbb{R}^n$,  $\max\limits_{i\in I_n}x_i(F_\mathcal{A}(\x))_i>0.$ Thus we have $$\alpha(F_\mathcal{A})=\min_{\|\x\|_\infty=1}\{\max_{i\in I_n}x_i(F_\mathcal{A}(\x))_i\}>0.$$

   If $\alpha(F_\mathcal{A})>0$, then it is obvious that for each nonzero $\y \in \mathbb{R}^n$,
   $$\max_{i\in I_n}\left(\frac{\y}{\|\y\|_\infty}\right)_i\left(F_\mathcal{A}(\frac{\y}{\|\y\|_\infty})\right)_i\geq\alpha(F_\mathcal{A}) >0.$$  Hence, there exists a $j\in I_n$ such that
   $$y_j(F_\mathcal{A}(\y))_j= y_j(\mathcal{A}\y^{m-1})_j^{\frac1{m-1}} >0.$$
 Thus $$y_j^{m-2}(y_j(\mathcal{A} \y^{m-1})_j)=y_j^{m-1}(\mathcal{A}\y^{m-1})_j >0.$$
 Since $m$ is even, we have $y_j^{m-2}>0$.   Hence,
   $y_j(\mathcal{A}\y^{m-1})_j>0$, i.e., $\mathcal{A}$ is a P tensor.
 \end{proof}
 %%%%%%%%%%%%%%%%%%%%%%%%%%%%%%%%%%%%%%%%%%%%%%%%%%%%%%%%%%%%%%%%%%%%%%%%%%%%%%%%%%%%%%%%%%%%%%%%%%%%%%%%%%%%%%A

\begin{Question}\label{Q2}\em For a P matrix $P$,  Mathias \cite{Ma} showed that $\alpha(A)$ has a strictly positive lower bound. Then for a P tensor $\mathcal{A} \in T_{m, n}$ ($m>2$), does $\alpha(F_\mathcal{A})$ or $\alpha(T_\mathcal{A})$ have a strictly positive lower bound?\end{Question}

\medskip

\begin{Question} \label{Q3}\em It is well-known that $A$ is a P matrix  if and only if the linear complementarity problem
$$\mbox{ find } \z \in \mathbb{R}^n\mbox{ such that }\z \geq \0, \q + A\z \geq \0, \mbox{ and }\z^\top (\q + A\z) = 0 $$
has a unique solution for all $\q \in \mathbb{R}^n$. Then for a P
tensor $\mathcal{A} \in T_{m, n}$ ($m>2$), does a similar property hold for
the following nonlinear complementarity problem
$$\mbox{ find } \x \in \mathbb{R}^n\mbox{ such that }\x \geq \0, \q + \mathcal{A}\x^{m-1} \geq \0, \mbox{ and }\x^\top (\q + \mathcal{A}\x^{m-1}) = 0? $$\end{Question}
%%%%%%%%%%%%%%%%%%%%%%%%%%%%%%%%%%%%%%%%%%%%%%%%%%%%%%%%%%%%%%%%%%%%%%%%%%%%%%%%%%%%%%%%%%%%%%%%%%%%%%%%%%%%%%

\section{B and B$_0$ Tensors}
\hspace{4mm}  An $n$-dimensional B matrix $B=(b_{ij})$ is a square
real matrix with its entries satisfying that for all $i\in I_n$
$$\sum_{j=1}^{n}b_{ij}>0\mbox{ and }\frac1n\sum_{j=1}^{n}b_{ij}>b_{ik},\ i\neq k.$$
 Many nice properties and applications of such B matrices have been studied by Pe\~na \cite{Pe,Pe1}. It was proved that B matrices are a subclass of P matrices in \cite{Pe}.

As a natural extension of B matrices, we now give the definition of
B and B$_0$ tensors.

\begin{Definition}\label{d3}\em Let $\mathcal{B} = (b_{i_1\cdots i_m}) \in T_{m, n}$.  We say that $\mathcal{B}$ is a B tensor iff for all $i \in I_n$
$$\sum_{i_2,\cdots, i_m=1}^{n}b_{ii_2i_3\cdots i_m}>0$$ and $$\frac1{n^{m-1}}\left(\sum_{i_2,\cdots, i_m=1}^{n}b_{ii_2i_3\cdots i_m}\right)>b_{ij_2j_3\cdots j_m} \mbox{ for all } (j_2, j_3, \cdots, j_m)\ne (i, i, \cdots, i).$$
We say that $\mathcal{B}$ is a B$_0$ tensor iff for all $i \in I_n$
$$\sum_{i_2,\cdots, i_m=1}^{n}b_{ii_2i_3\cdots i_m}\ge0$$ and $$\frac1{n^{m-1}}\left(\sum_{i_2,\cdots, i_m=1}^{n}b_{ii_2i_3\cdots i_m}\right)\ge b_{ij_2j_3\cdots j_m} \mbox{ for all } (j_2, j_3, \cdots, j_m)\ne (i, i, \cdots, i).$$
\end{Definition}

Unlike P and P$_0$ tensors, it is easily checkable if a given tensor
in $T_{m, n}$ is a B or $B_0$ tensor or not.

\subsection{Entries of B and B$_0$ Tensors} \hspace{4mm}

We first study some properties of entries of B and B$_0$ tensors.

\begin{Theorem}\label{t5}\em Let $\mathcal{B} = (b_{i_1\cdots i_m}) \in T_{m, n}$.  If $\mathcal{B}$ is a B tensor, then for each $i\in I_n$,
 $$b_{ii\cdots i}>\max\{0, b_{ij_2j_3\cdots j_m}; (j_2, j_3, \cdots, j_m)\ne (i, i, \cdots, i), j_2,j_3,\cdots, j_m \in I_n \}.$$
 If $\mathcal{B}$ is a B$_0$ tensor, then for each $i\in I_n$,
 $$b_{ii\cdots i}\ge \max\{0, b_{ij_2j_3\cdots j_m}; (j_2, j_3, \cdots, j_m)\ne (i, i, \cdots, i), j_2,j_3,\cdots, j_m \in I_n \}.$$
 \end{Theorem}

  \begin{proof}
   Suppose that $\mathcal{B} \in T_{m, n}$ is a B tensor.   By Definition \ref{d3} that for all $i \in I_n$
\begin{equation}\label{eq:31}\sum_{i_2,\cdots, i_m=1}^{n}b_{ii_2i_3\cdots i_m}>0\end{equation} and
\begin{equation}\label{eq:32}\frac1{n^{m-1}}\left(\sum_{i_2,\cdots, i_m=1}^{n}b_{ii_2i_3\cdots i_m}\right)>b_{ij_2j_3\cdots j_m} \mbox{ for all }
(j_2, j_3, \cdots, j_m)\ne (i, i, \cdots, i).\end{equation} Let
$b_{ik_2k_3\cdots k_m}=\max\{b_{ij_2j_3\cdots j_m}: (j_2, j_3,
\cdots, j_m)\ne (i, i, \cdots, i)\}.$ Then it follows from
(\ref{eq:32}) that $$\sum_{i_2,\cdots, i_m=1}^{n}b_{ii_2i_3\cdots
i_m}>n^{m-1}b_{ik_2k_3\cdots k_m}\geq b_{ik_2k_3\cdots
k_m}+\sum_{(j_2, j_3, \cdots, j_m)\ne (i, i, \cdots, i) }
b_{ij_2j_3\cdots j_m}.$$  Thus  $$b_{iii\cdots i}>b_{ik_2k_3\cdots
k_m}=\max\{b_{ij_2j_3\cdots j_m}: (j_2, j_3, \cdots, j_m)\ne (i, i,
\cdots, i)\}.$$ Therefore, $ b_{iii\cdots i}>0$. In fact, suppose
$b_{iii\cdots i}\leq0.$ Then $\max\{b_{ij_2j_3\cdots j_m}: (j_2,
j_3, \cdots, j_m)\ne (i, i, \cdots, i)\}< b_{iii\cdots i}\leq0$,
which implies that
$$\sum_{i_2,\cdots, i_m=1}^{n} b_{ii_2i_3\cdots
i_m}\leq0.$$ This contradicts to (\ref{eq:31}).    The case for
B$_0$ tensors can be proved similarly.
\end{proof}
%%%%%%%%%%%%%%%%%%%%%%%%%%%%%%%%%%%%%%%%%%%%%%%%%%%%%%%%%%%%%%%%%%%%%%%%%%%%%%%%%%%%%%%%%%%%%%%%%%%%%%%%%%%%%
Let $\mathcal{B} = (b_{i_1\cdots i_m}) \in T_{m, n}$.  For each $i\in I_n$,
define
\begin{equation}\label{eq:33}\beta_i(\mathcal{B})=\max\{0,
b_{ij_2j_3\cdots j_m}; (j_2, j_3, \cdots, j_m)\ne (i, i, \cdots, i),
j_2,j_3,\cdots, j_m \in I_n \}.\end{equation} With the help of the
quantity $\beta_i(\mathcal{B})$, we will study further the properties of B
tensors.

 \begin{Theorem}\label{p3}\em Let $\mathcal{B} = (b_{i_1\cdots i_m}) \in T_{m, n}$.  Then $\mathcal{B}$ is B tensor if and only if for each $i\in I_n$,
 $$\sum_{i_2,\cdots, i_m=1}^{n} b_{ii_2i_3\cdots
 i_m}>n^{m-1}\beta_i(\mathcal{B});$$  and $\mathcal{B}$ is B$_0$ tensor if and only if for each $i\in I_n$,
 $$\sum_{i_2,\cdots, i_m=1}^{n} b_{ii_2i_3\cdots
 i_m}\ge n^{m-1}\beta_i(\mathcal{B}).$$
 \end{Theorem}

  \begin{proof}
   Since $\beta_i(\mathcal{B})\geq0$, the desired conclusion directly follows from Definition \ref{d3}.   \end{proof}

%%%%%%%%%%%%%%%%%%%%%%%%%%%%%%%%%%%%%%%%%%%%%%%%%%%%%%
 \begin{Theorem}\label{t6}\em Let $\mathcal{B} = (b_{i_1\cdots i_m}) \in T_{m, n}$.   If $\mathcal{B}$ is a B tensor, then for each $i\in I_n$, \begin{itemize}
\item[(i)] $b_{ii\cdots i}>\sum\limits_{b_{ii_2\cdots i_m}<0} |b_{ii_2i_3\cdots i_m}|$;
\item[(ii)] $b_{ii\cdots i}> |b_{ij_2j_3\cdots j_m}|$ for all $(j_2, j_3, \cdots, j_m)\ne (i, i, \cdots, i), j_2,j_3,\cdots, j_m \in
I_n$.
\end{itemize}
If $\mathcal{B}$ is a B$_0$ tensor, then (i) and (ii) hold with ``$>$'' being
replaced by ``$\ge$''.
 \end{Theorem}

  \begin{proof}  Suppose that  $\mathcal{B}$ is a B tensor.
  (i)  It follows from Proposition \ref{p3} that for each $i\in I_n$
\begin{equation}\label{eq:34}b_{ii\cdots i}-\beta_i(\mathcal{B})>\sum_{(j_2, j_3, \cdots, j_m)\ne (i, i, \cdots, i)}(\beta_i(\mathcal{B})-b_{ij_2j_3\cdots j_m}).\end{equation}
It follows from Definition \ref{d3} that for all $i\in I_n$,
$$ \beta_i(\mathcal{B})\geq0\mbox{ and }
\beta_i(\mathcal{B})-b_{ij_2j_3\cdots j_m}\geq0\mbox{ for all } (j_2, j_3,
\cdots, j_m)\ne (i, i, \cdots, i).$$ Then for all $b_{ii_2i_3\cdots
i_m}<0$,
$$\beta_i(\mathcal{B})-b_{ii_2i_3\cdots i_m}\geq
|b_{ii_2i_3\cdots i_m}|$$ and $$b_{ii\cdots i}\geq b_{ii\cdots
i}-\beta_i(\mathcal{B}). $$ So  by (\ref{eq:34}),  we have
$$b_{ii\cdots i}>\sum_{(j_2, j_3, \cdots, j_m)\ne (i, i, \cdots, i)}(\beta_i(\mathcal{B})- b_{ij_2j_3\cdots j_m})\geq\sum\limits_{b_{ii_2\cdots i_m}<0} |b_{ii_2i_3\cdots i_m}|.$$

(ii) is an obvious conclusion by combining Theorem \ref{t5} with
(i).

The case for B$_0$ tensors can be proved similarly.
 \end{proof}

\subsection{Principal Sub-Tensors of B and B$_0$ Tensors} \hspace{4mm}

We now show that every principal sub-tensor of a B (B$_0$) tensor is
a B (B$_0$) tensor.

 \begin{Theorem}\label{t7}\em Let $\mathcal{B} = (b_{i_1\cdots i_m}) \in T_{m, n}$.   If $\mathcal{B}$ is a B (B$_0$) tensor, then every principal sub-tensor of $\mathcal{B}$ is also a B (B$_0$) tensor.
  \end{Theorem}

  \begin{proof}   Suppose that $\mathcal{B}$ is a B tensor.
   Let a principal sub-tensor $\mathcal{B}^J_r$ of $\mathcal{B}$ be given.  Then it follows from Theorem \ref{t6} (i) that for all $i\in J,$ $$\sum_{i_2,\cdots, i_m\in J}b_{ii_2i_3\cdots i_m}>0.$$
   Now it suffices to show that  for all $i\in J,$ $$\sum_{i_2,\cdots, i_m\in J}b_{ii_2i_3\cdots i_m}>r^{m-1}b_{ij_2\cdots j_m} \mbox{ for all }
   (j_2, j_3, \cdots, j_m)\ne (i, i, \cdots, i),\  j_2,\cdots, j_m\in J.$$
Suppose not. Then there is $(i, j_2, \cdots, j_m)$ such that $i, j_2, \cdots, j_m \in J$ and $$\sum_{i_2,\cdots,
i_m\in J}b_{ii_2i_3\cdots i_m}\leq r^{m-1}b_{ij_2\cdots j_m}.$$ Let $b_{ik_2k_3\cdots
k_m}=\max\{b_{ii_2i_3\cdots
i_m};(i_2, i_3, \cdots, i_m) \ne (i, i, \cdots, i) \mbox{ and }i_2, i_3, \cdots, i_m\in I_n\}$. Then $b_{ik_2k_3\cdots
k_m}\geq b_{ij_2\cdots j_m}$.   Hence,
$$\begin{aligned}
  n^{m-1}b_{ik_2k_3\cdots k_m}\geq&r^{m-1}b_{ik_2k_3\cdots k_m}+\sum \left\{ b_{ii_2i_3\cdots i_m} : {\rm \ not \ all\ of\ } i_2, \cdots, i_m\ {\rm are\ in\ } J \right\}\\
  \geq&r^{m-1}b_{ij_2j_3\cdots j_m}+\sum \left\{ b_{ii_2i_3\cdots i_m} : {\rm \ not \ all\ of\ } i_2, \cdots, i_m\ {\rm are\ in\ } J \right\}\\
  \geq&\sum_{i_2,\cdots, i_m\in J}b_{ii_2i_3\cdots i_m}+\sum \left\{ b_{ii_2i_3\cdots i_m} : {\rm \ not \ all\ of\ } i_2, \cdots, i_m\ {\rm are\ in\ } J \right\}\\
  =& \sum_{i_2,\cdots, i_m=1}^n b_{ii_2i_3\cdots i_m}.
  \end{aligned}$$
Thus $$ \frac1{n^{m-1}}\left(\sum_{i_2,\cdots,
i_m=1}^n b_{ii_2i_3\cdots
i_m}\right)\leq b_{ik_2k_3\cdots k_m},$$ which obtains a
contradiction since $\mathcal{B}$ is a  B tensor.

The case for B$_0$ tensors can be proved similarly.
  \end{proof}

\subsection{The Relationship with M Tensors} \hspace{4mm}

Recall  that a tensor $\mathcal{A} = (a_{i_1\cdots i_m}) \in T_{m,
n}$ is called a Z tensor iff all of its off-diagonal entries are
non-positive, i.e., $a_{i_1\cdots i_m} \le 0$ when never $(i_1,
\cdots, i_m) \not = (i, \cdots, i)$ \cite{ZQZ}; $\mathcal{A}$ is called diagonally
dominated iff for all $i \in I_n$,
$$a_{i\cdots i} \ge \sum \{ |a_{ii_2\cdots i_m}| : (i_2,
\cdots, i_m) \not = (i, \cdots, i), i_j \in I_n, j = 2, \cdots, m
\};$$ $\mathcal{A}$ is called strictly diagonally dominated iff for all $i \in
I_n$,
$$a_{i\cdots i} > \sum \{ |a_{ii_2\cdots i_m}| : (i_2,
\cdots, i_m) \not = (i, \cdots, i), i_j \in I_n, j = 2, \cdots, m
\}.$$    It was proved in \cite{ZQZ} that a diagonally dominated Z
tensor is an M tensor, and a strictly diagonally dominated Z tensor
is a strong M tensor.   The definition of M tensors may be found in
\cite{ZQZ, DQW}.   Strong M tensors are called nonsingular tensors
in \cite{DQW}.   Laplacian tensors of uniform hypergraphs, defined
as a natural extension of Laplacian matrices of graphs, are M
tensors [16, 20-23].% \cite{Qi14, HQ14, HQS, HQX, QSW}.

Now we give the properties of a B (B$_0$) tensor under the condition
that it is a  Z tensor.

\begin{Theorem}\label{t 7} \em Let $\mathcal{B}=(b_{i_1i_2i_3\cdots i_m}) \in T_{m, n}$ be a Z tensor. Then the following properties are equivalent:\begin{itemize}
\item [(i)] $\mathcal{B}$ is  B (B$_0$) tensor;
\item [(ii)]  for each $i \in n$,  \ \ $\sum\limits_{i_2,\cdots, i_m=1}^n b_{ii_2i_3\cdots
i_m}$ is positive (nonnegative);
\item [(iii)] $\mathcal{B}$ is strictly diagonally dominant (diagonally dominated).
\end{itemize}
\end{Theorem}
\begin{proof}    We now prove the case for B tensors.   The proof for the B$_0$ tensor case is similar.

It follows from Definition \ref{d3} that (i) implies (ii).

Since  $\mathcal{B}$ be a  Z tensor, all of its off-diagonal entries are
non-positive.   Thus, for any of its off-diagonal entry
$b_{ii_2\cdots i_m}$, we have $\left|b_{ii_2i_3\cdots
i_m}\right|=-b_{ii_2i_3\cdots i_m}$.    Thus, (ii) means that for $i
\in I_n$,
$$\begin{aligned}
b_{iii\cdots i} > & -\sum \left\{ b_{ii_2i_3\cdots i_m} : (i_2,
\cdots, i_m) \ne (i, \cdots, i), i_j \in I_n, j= 2, \cdots, m
\right\}\\
= & \sum \left\{ |b_{ii_2i_3\cdots i_m}| : (i_2, \cdots, i_m) \ne
(i, \cdots, i), i_j \in I_n, j= 2, \cdots, m \right\}\\
\geq & \ 0.
\end{aligned}$$    Thus, (ii) implies (iii).

From (iii), it is obvious that  for each $i \in I_n$,
$$\sum\limits_{i_2,\cdots, i_m=1}^n b_{ii_2i_3\cdots
i_m}>0.$$    Since all the off-diagonal entries of $\mathcal{B}$ are
non-positive, we have
$$\frac1{n^{m-1}}\sum\limits_{i_2,\cdots,
i_m=1}^n b_{ii_2i_3\cdots i_m}>0\geq b_{ii_2i_3\cdots i_m} \mbox{
for all }(i_2, \cdots, i_m) \ne (i, \cdots, i).$$ This shows that
(iii) implies (i). \end{proof}

From this theorem, we see that if a Z tensor is also a B$_0$ (B)
tensor, then it is a (strong) M tensor.

\bigskip

\begin{Question}\label{Q4}\em When $m=2$, it is known that each B matrix is a P
matrix.   If $m$ is odd, in general, a B (B$_0$) tensor is not a P
(P$_0$) tensor.   For example, let $a_{i\cdots i} = 1$ and
$a_{i_1\cdots i_m} = 0$ otherwise.  Then $\mathcal{A} = (a_{i_1\cdots i_m})$
is the identity tensor.   When $m$ is odd, the identity tensor is a
$B$ tensor, but not a P or P$_0$ tensor.    But we still make ask,
when $m \ge 4$ and is even, is a B (B$_0$) tensor a P (P$_0$)
tensor?\end{Question}

\begin{Question}\label{Q5}\em A symmetric P (P$_0$) tensor is positive
(semi-)definite.   When $m \ge 4$ and is even, is a symmetric B
(B$_0$) tensor positive (semi-)definite?   If the answer is ``yes''
to this question, then we will have another checkable sufficient
condition for positive (semi-)definite tensors.\end{Question}

\begin{Question}\label{Q6}\em What are the spectral properties of a B (B$_0$)
tensor?\end{Question}

\section{Perspectives}%Conclusions and Remarks}
\hspace{4mm}  In this paper, we make an initial study on P, P$_0$, B
and B$_0$ tensors.    We see that they are linked with positive
(semi-)definite tensors and M tensors, which are useful in
automatical control, magnetic resonance imaging and spectral
hypergraph theory.   The six questions at the ends of Sections 3-5
pointed out some further research directions.

In the following, we point out the  potential links between the above results and
optimization, nonlinear equations, nonlinear complementarity
problems, variational inequalities and the non-negative tensor
theory.

\medskip

{\bf (i)}.  The potential link between the above results and optimization, nonlinear
equations, nonlinear complementarity problems and variational
inequalities.   Question \ref{Q3} has also pointed out the potential link between P tensor and nonlinear complementarity problems.  We may also consider the
optimization problem
$$\min \{ \mathcal{A} \x^m + \q^\top \x \},$$
the nonlinear equations \cite{OR}
$$\mathcal{A} x^{m-1} = \q$$
and the variational inequality problem \cite{FaP, QSZ}
$$\mbox{ find } \x_* \in X, \mbox{such that} (\x - \x^*)^\top \mathcal{A}
\x_*^{m-1} \ge 0, \mbox{ for all } \x \in X,$$ where $X$ is a
nonempty closed subset of $\mathbb{R}^n$.  When $\mathcal{A}$ is a P, P$_0$, B or
B$_0$ tensor, what properties we can obtain for the above problems?

{\bf (ii)}.  The potential link between the above results and the non-negative tensor theory.
The non-negative tensor theory at least include two parts:  the
non-negative tensor decomposition \cite{CZPA} and the spectral theory
of non-negative tensors \cite{CQZ}.   The recent paper \cite{QXX}
linked these two parts.   However, there are still many questions
not answered in non-negative tensors.   In the non-negative matrix
theory \cite{BP}, a doubly non-negative matrix is a positive
semi-definite, non-negative matrix.  The research on positive
semi-definite, non-negative tensors is very little.   Thus, we may
ask a question weaker than Question \ref{Q5}:

\medskip

\begin{Question}\label{Q7}\em When $m \ge 4$ and is even, is a symmetric
nonnegative B (B$_0$) tensor positive (semi-)definite?   If the
answer is ``yes'' to this question, then we will have more
understanding on positive semi-definite, non-negative tensors.\end{Question}

We may also ask the following question:

\begin{Question}\label{Q8}  \em What is the relation between non-negative B (B$_0$)
tensors and completely positive tensors introduced in \cite{QXX}?\end{Question}

In a word, this paper is only an initial study on P, P$_0$, B and
B$_0$ tensors.   Many questions for these tensors are waiting for
answers.

It should be pointed out that after the first version of this paper, two more papers [35, 36] on P, P$_0$, B and B$_0$ tensors appeared.    In [35], it was proved that an even order
symmetric B$_0$ tensor is positive semi-definite and an even order symmetric B tensor is positive definite.   Some further properties of P, P$_0$, B and B$_0$ tensors were obtained in [36].  These answered some questions raised in this paper and enriched the theory of P, P$_0$, B and B$_0$ tensors.

%%%%%%%%%%%%%%%%%%%%%%%%%%%%%%%%%%%%%%%%%%%%%%%%%%%%%%%%%%%%%%%%%%%%%%%%%%%%%%%%%%%%%%%%%%%%%%%%%%%%%%%%%%%%%%%%%%%%%%%%%%%%%%%%%%%%%%%%%%%%%%%%%%%%%%%%%%%%
\section{Conclusions}

In this paper, we extend some classes of structured matrices to higher order tensors.  We discuss their relationships with positive semi-definite tensors and some other
  structured tensors.    We show that every principal sub-tensor of such a structured tensor is still a structured tensor in the same class, with a lower dimension.
  The potential links and applications of such structured tensors are also discussed.

There are more research topics on structured tensors.   In particular, can one construct an efficient algorithm to compute the extreme eigenvalues of a special structured tensor, other than the largest eigenvalue of a nonnegative tensor?
It is well-known \cite{CQZ} that there are efficient algorithms for computing the largest eigenvalue of a nonnegative tensor.   Until now, there are no polynomial-time algorithms for
computing extreme eigenvalues of structured tensors in the other cases.    The first challenging problem is to construct an efficient algorithm to compute the smallest real eigenvalue of
a Hilbert tensor \cite{SQ1}, with the condition that such a real eigenvalue has a real eigenvector.    A further challenging problem is to address the above problem for a Cauchy tensor \cite{CQ}
instead of a Hilbert tensor.  Note that the Hilbert tensor is
a special case of the Cauchy tensor \cite{CQ}.

\section*{\bf Acknowledgment}
The authors would like to thank the anonymous
referees for their valuable suggestions which helped us to improve
this manuscript.

\end{document}